\documentclass{article}
\usepackage{amsmath,amsfonts,amssymb,amsthm,bbm,latexsym,mathrsfs}
\usepackage{graphicx,color,epsfig,fancyhdr,dsfont}
\usepackage{enumerate}
\usepackage{hyperref}
\usepackage{indentfirst}
\usepackage{amsmath,amscd}
\usepackage{caption}
\usepackage{graphicx, subfig}
\usepackage[all]{xy}
\usepackage{stmaryrd}
\usepackage{extarrows}

\newtheorem{definition}{Definition}
\newtheorem{lemma}[definition]{Lemma}
\newtheorem{theorem}[definition]{Theorem}
\newtheorem{coro}[definition]{Corollary}
\newtheorem{prop}[definition]{Proposition}
\newtheorem{example}[definition]{Example}
\newtheorem{remark}[definition]{Remark}
\newtheorem{fact}[definition]{Fact}
\newtheorem{ques}[definition]{Question}
\newtheorem{conj}[definition]{Conjecture}

\newtheorem*{claim}{Claim}

\newenvironment{claimproof}[1][\proofname]
               {
                 \proof[#1]
                 
               }
               {
                 \endproof
               }

\newcommand{\N}{\mathbb{N}}

\newcommand{\Z}{\mathbb{Z}}

\newcommand{\Acl}{\mathrm{Acl}}

\newcommand{\R}{\mathbb{R}}

\newcommand{\rk}{\mathrm{rk}}

\newcommand{\J}{\mathcal{J}}
\newcommand{\I}{\mathcal{I}}

\newcommand{\HH}{\mathrm{Haar}}
\newcommand{\Sym}{\mathrm{Sym}}
\newcommand{\supp}{\mathrm{supp}}
\newcommand{\SO}{\mathrm{SO}}

\newcommand{\dd}{\mathrm{d}}

\newcommand{\RP}{\mathbb{RP}}

\newcommand{\ind}{\mathrel{\raise0.2ex\hbox{\ooalign{\hidewidth$\vert$\hidewidth\cr\raise-0.9ex\hbox{$\smile$}}}}}

\begin{document}
\title{Independence relations induced by ideals}
\author{Zhentao Zhang}
\date{}
\maketitle

\begin{abstract}
Inspired by the non-meager independence $\ind^{nm}$ introduced by Krupi\'{n}ski \cite{KP}, we study further possible independence relations induced by ideals and provide a general framework for this topic. We show that the independence relation $\ind^\HH$ induced by Haar null ideals is a good example for locally compact groups and provide an example showing $\ind^\HH\neq \ind^{nm}$.
Moreover, we introduce an order on the collection of all well-behaved independence relations and conjecture that $\ind^{nm}$ is the maximum one for Polish groups. We prove the conjecture  under the extra hypothesis of $\sigma$-compactness. For Lie groups, we discuss $\SO(3)$ and 
$\mathrm{SE}(2)$ as examples, and prove that the independence relation is unique for nilpotent Lie groups. 
\end{abstract}

\section{Introduction}

In \cite{KP}, Krupi\'{n}ski introduced  a notion of Polish structure, and provided  a setting for developing an analog of geometric stability from model theory. This is a generalization of the profinite structures studied by Newelski, and the core idea originates from series of works by Newelski (see, e.g., \cite{L}). Further related references can be found in \cite{KP}. 

Let $G$ be a topological group. Unless stated otherwise, $G$ is Hausdorff.  Let $(G,M)$ be an action $G\times M\rightarrow M: (g,a)\rightarrow ga$. We naturally extend the action on $M^n$ by $g(a_1,\dots ,a_n)=(ga_1,\dots,g a_n)$ for each $n\in\N$.
For a subset $A$ of $M$ or a sequence $a=(a_i)$ of $M$, we let $G_A$ (resp. $G_a$) be the subgroup of $G$ which fixes $A$ (resp. $a$) pointwise. 
We call $M$ a \textbf{$G$-structure} if $G_a$ is closed for each $a\in M$ (equivalently, for each $a\in M^n$). 

For $a\in X^n$ and $A\subset M$, the orbit of $a$ over $A$, defined by $o(a/A):=\{ga:g\in G_A\}$ is treated as the type of $a$ over $A$. So for $b\in o(a/A)$ and $B\supset A$, we call $o(b/B)$ an extension of $o(a/A)$ over $B$, although $o(b/B)\subset o(a/A)$. In analogy with the independence of type extensions in model theory, we study the independence of orbit extensions. In Newelski's setting, $M$ is a compact topological space with a continuous 
$G$-action, and orbit independence is characterized by the meager property of extensions. A generalization by Krupi\'{n}ski removes the need for a topology on $M$, pulling the entire discussion of independence back to the group $G$. 

In \cite{KP}, nm-independence and o-dependence are studied.  For $a$, $b$, $C$ be finite subsets or tuples of $M$, we say that \textbf{$a$ is mn-independent (or o-independent) with $b$ over $C$}, and write $a\ind_C^{nm}b$ (resp. $a\ind_C^o b$) if $G_{Ca} G_{Cb}$ is non-meager (resp. open) in $G_C$ (\cite{KP} Proposition 2.3). Here, $Ca$ denotes $C\cup a$, which does not cause confusion as no operation is studied on $M$ in this note. 
The crucial fact is that these independence notions obey the independence axioms when $G$ is Polish.

\begin{fact}[\cite{KP} Theorem 2.5]\label{Axioms}
Let $\ind=\ind^{nm}$ or $\ind^o$. Let $G$ be Polish. Then 
\begin{itemize}
\item (Invariance) $a\ind _C b$ iff $g(a) \ind_{g(C)} g(b)$ for every finite $a,b, C\subset M$ and $g\in G$.

\item (Symmetry) $a\ind_C b$ iff $b\ind_C a$ for every finite $a,b,C\subset M$.

\item (Transitivity) $a\ind_B C$ and $a\ind_A B$ iff $a\ind_A C$ for every finite $a\subset M$ and finite $A\subset B\subset C\subset M$.
\end{itemize}
\end{fact}

As model theory emphasizes theories over particular models, this note focuses on the group $G$ and its related independence relation, treated as a kind of geometry. 
In section \ref{section-1}, we outline the general framework and explain why we set 
$M$ aside for convenience. 

In section \ref{section-2}, we focus on the independence relation induced by measures. We  show that for locally compact $G$, the independence relation induced 
by the Haar measure, denoted by $\ind^\HH$, is an important example. We present an example of a compact Polish group $G$ for which $\ind^{mn}\neq \ind^\HH$. Moreover, we try to generalize the Haar measure null ideal to the Haar null ideal for non-locally compact Polish groups. Unfortunately, there is an example showing that the induced independence relation does not satisfy the required axioms. On the positive side, we show that, for commutative Polish groups, the independence relation induced by the Haar null ideal coincides with $\ind^{mn}$ and $\ind^o$.

In section \ref{section-3}, we  try to classify all independence relations on certain groups in our framework. Moreover, we will give an order on the collection of the independence relations for $G$, based on their induced ranks. We conjecture that $\ind^{nm}$ is the maximum independence relation when $G$ is Polish. We prove the conjecture under the additional assumption that $G$ is $\sigma$-compact. We also give a more detailed study for the case where $G$ is a Lie group. We We obtain a complete classification of all independence relations on the special orthogonal group $\SO(3)$ and the special Euclidean group of the plane $\mathrm{SE}(2)$. Moreover, we show the uniqueness of the independence relation when $G$ is a nilpotent Lie group.

\section{Ideal-equipped group}\label{section-1}

Let $G$ be a topological group and $M$ a $G$-structure. For both mn-independence and o-independence, given finite $a,b,C\subset M$, the independence of $a$ and $b$ over $C$ can be intuitively interpreted as $G_{Ca} G_{Cb}$ being not ``small'' in $G_C$. Now  We  clarify the notion of ``small'' in terms of ideals.
Because ``smallness'' is to be discussed over all possible $G_C$, we will actually equip each closed subgroup of $G$ with an ideal. 
An \textbf{ideal scheme} of $G$ is a family of $\I=\{\I_H\}_H$ where $H$ runs over all closed subgroups of $G$ and each $\I_H$ is an ideal of $H$. We call a topological group $G$ equipped with an ideal scheme $\I$ an \textbf{ideal-equipped group}.

Then we define the $\I$-independence as follows:  for finite $a,b, C\subset M$, we write $a\ind^\I_C b$  iff $G_{Ca} G_{Cb}\notin \I_{G_C}$.
It is clear that $\ind^{nm}$ is the case in which $\I_H$ is the meager ideal on $H$ for every closed subgroup $H$. Although $\ind^o$ does not arise directly in this way, in the environment where $\ind^o$ is studied, in most environments  in which o-independence is studied (when $G$ is compact and 
$M$ is a compact space with a continuous 
$G$-action), $\ind^o=\ind^{nm}$ (\cite{KP} Corollary 2.13).

But in general $\ind^\I$ does not satisfy the properties (or axioms) listed in Fact \ref{Axioms}. 
Let us deal with them.

$\bullet$ (Symmetry)
It is obvious from the definition that Symmetry always holds. 

$\bullet$ (Invariance) For Invariance, we assume that every $\I_H$ is $H$-invariant. When $(G, \ind^\I)$ is viewed as a kind of geometry, this requirement is entirely natural. Furthermore, the independence relation should be invariant under inner automorphisms, which leads us to the additional assumption that $\I_{g^{-1}Hg} = g^{-1} \I_H g:=\{g^{-1}Xg:X\in\I_H\}$ for every closed subgroup $H$ and $g\in G$.

$\bullet$ (Transitivity) Transitivity is the central issue. Let us reduce it to a condition on $G$. If we name $G_{Aa}$ by $H$, $G_A$ by $K_1$,
$G_B$ by $K_2$,$G_C$ by $K_3$, we have $K_1>K_2>K_3$ and $H<K_1$. The Transitivity axiom translates into the statement that $(H\cap K_2)K_3\notin \I_{K_2}$ and $H K_2\notin \I_{K_1}$ iff $HK_3\notin \I_{K_1}$.

Now we consider some additional conditions.

$\bullet$ (Existence of extensions) Here we consider the existence of independent extensions of orbits. Recall that the $G$-structure $M$ is \textbf{small} if the action on each $M^n$ has only countably many orbits. 
In this case, for finite subsets $A\subset B$ and $a\in M^n$, the orbit  $o(a/A)$ always admits an nm-independent extension over $B$, i.e., there exists $b\in o(a/A)$ such that $b\ind^{nm}_A B$ (\cite{KP} Theorem 2.10). The essential reason is that the meager ideal is a $\sigma$-ideal. We therefore upgrade the ideal scheme to a 
$\sigma$-ideal scheme, which means that every $\I_H$ is a $\sigma$-ideal. We carry over the proof of nm-independence to the setting of $\I$-independence; the argument is entirely similar.

\begin{prop}
Let $\I$ be a $\sigma$-ideal scheme with Invariance. Assume that $M$ is a small $G$-structure. Then for finite subsets $A\subset B$ and $a\in M^n$, the orbit  $o(a/A)$ admits an $\I$-independent extension over $B$. 
\end{prop}
\begin{proof}
By smallness, There is a countable set $\{a_i\}_i$ of representatives of all orbits over $B$ contained in $o(a/A)$. Take $g_i\in G_A$ such that $a_i=g_i a$. Then it is easy to check that $G_{Aa_i}=g_i G_{Aa} g_i^{-1}$, so $a_i\ind_A B$ iff $G_B g_iG_{Aa} g_i^{-1}\notin \I_{G_A}$ iff (by Invariance) $G_B g_iG_{Aa}\notin \I_{G_A}$. Also, it is easy to check that $G_A=\bigcup_i G_B g_i G_{Aa}$. Since $\I_{G_A}$ is a $\sigma$-ideal, there is $i$ such that $G_B g_iG_{Aa}\notin \I_{G_A}$, which implies $a_i\ind^\I_A B$.
\end{proof}

$\bullet$ (Dichotomy for closed subgroups)
Recall that for a finite set $A\subset M$, $\Acl(A)$ is defined as the set of all elements with countable orbits over $A$. Apart from the three axioms we have mentioned in Fact \ref{Axioms}, \cite{KP} Theorem 2.5 also addresses the relationship between $\Acl$
and $\ind$: $a\in \Acl(A)$ iff $a\ind_A B$ for every finite set $B\supset A$. In the language of ideal-equipped groups, this property writes: for closed subgroups $K<H$, $K \in \I_H$ iff 
$H/K$ is uncountable.

We now summarize the properties discussed above.

\begin{definition}\label{admissible}
An  ideal scheme $\I$  of $G$ is \textbf{admissible} if it satisfies
\begin{itemize}
\item ($\sigma$-ideal) Every $\I_H$ is a $\sigma$-ideal.

\item (Invariance) 
Every $\I_H$ is $H$-invariant; $g^{-1}\I_H g=\I_{g^{-1}Hg}$ for every $g\in G$.
\item (Transitivity) For closed subgroups $K_1>K_2>K_3$ and $H<K_1$,$(H\cap K_2)K_3\notin \I_{K_2}$ and $H K_2\notin \I_{K_1}$ iff $HK_3\notin \I_{K_1}$.
\end{itemize}

The admissible scheme $\I$ is  \textbf{strongly admissible} if it further satisfies the following property:
\begin{itemize}
\item (Dichotomy for closed subgroups) For closed subgroups $K<H$, $K \in \I_H$ iff 
$H/K$ is uncountable.
\end{itemize}
\end{definition}

It is clear that for every closed subgroup $L$ of $G$, the restriction $\I|_L:=\{\I_{H\cap L}\}_H$ is an ideal scheme of $L$. Moreover, if $\I$ is (strongly) admissible, then so is $\I|_L$.

The reason we formulate these properties on the group $G$ is that we want to take into account arbitrary $G$-structures instead of a particular $M$. This is analogous to the model-theoretic practice of working with sufficiently saturated models. From the above discussion, we know that when $\I$ is admissible, the induced independence relation $\ind^\I$ satisfies the axioms listed in  Fact \ref{Axioms}. Note that, for a $\sigma$-ideal scheme $\I$, a failure of admissibility can be witnessed by finitely many closed subgroups $H_i$, then the failure of axioms for $\ind^\I$ will be witnessed by the $G$-structure $M:=\bigsqcup_i G/H_i$, which is small and faithfully acted by $G$.

If the ideal scheme $\I$ is {admissible geometry} (resp. {strongly admissible geometry}), we call $(G,\ind^\I)$ an \textbf{admissible geometry} (resp. \textbf{strongly admissible geometry}) for $G$. We care about the geometry, not the ideal scheme $\I$, although the latter determines the former.
Obviously, $\I^\emptyset:=\{\I_H^\emptyset=\{\emptyset\}\}_H$ gives an admissible geometry. However, strongly admissible geometry may not exist. 

\begin{example}
Let $G=\sum_{m<\omega} (\R,+)$ be equipped with the discrete topology. Let $H_n=\sum_{m\leq n} \R$. If $\I$ is strongly admissible, then $H_n\in \I_G$. However, $G=\bigcup_{n\in \omega} H_n$ which is impossible.
\end{example}

We call a topological group $G$ \textbf{tame} if it admits at least one strongly admissible geometry. 
Of course, when $G$ is Polish, $(G,\ind^{nm})$ is a strongly admissible geometry. So Polish groups are tame.

Here is a natural but impractical  question:

\begin{ques}
Given a topological group $G$, classify all (strongly) admissible geometries for $G$.
\end{ques}

We aim to approach this problem by restricting $G$ to a suitable class of groups, or imposing further natural conditions on $\I$, in order to obtain a definitive answer. This parallels the model-theoretic characterization of forking independence in simple theories (\cite{Simple} Theorem 6.2.1). We will give some related results, but we are still far from this goal.

\section{Ideal schemes from measures}\label{section-2}

We now study ideal schemes arising from measures, which provide good examples. 

Let us first assume $G$ to be locally compact and consider its (left) Haar measure $\mu_G$. In addition, we may assume that $G$ is $\sigma$-compact, i.e., $G$ can be covered by countably many compact subsets. Note that every locally compact group has an open $\sigma$-compact subgroup, so we can establish the theory for this subgroup and then extend it to the whole group by translation.  We will review some results for Haar measures, for which  \cite{Book} Chapter 2 is a good reference. 

Let $H$ be a closed subgroup of $G$. It is clear that $H$ is locally compact and admits a (left) Haar measure $\mu_H$. A Borel subset $X$ of $H$ is \textbf{null} if $\mu_H(X)=0$. A subset of $H$ is \textbf{null} if it is contained in a Borel null set. Note that the notion of a null set does not depend on the choice of Haar measure, nor on whether it is left or right Haar measure.
It is clear that the null subsets of $H$ form a $\sigma$-ideal $\I^\HH_H$ and $\I^\HH:=\{\I_H^\HH\}$ is an ideal scheme on $G$.

We can also discuss the quotient space $G/H$.
There is a rho-function $\rho:G\rightarrow \R^{>0}$  such that $\rho(gh)=\frac{\Delta_H(h)}{\Delta_G(h)}\rho(g)$ for $g\in G,h\in H$ where $\Delta_G$ and $\Delta_H$ are modular functions of $G$ and $H$ respectively. The rho-function gives a strongly quasi-invariant measure $\mu_{G/H}$ on $G/H$ with \textit{Weyl integration formula}: 
$$\int_{G/H} (\int_H f(xy) \dd \mu_H( y)) \dd \mu_{G/H}(xH)=\int_G f(x) \dd \mu_G(x)$$
for each compactly supported continuous function $f:G\rightarrow \R$. Again, a Borel subset $X$ of $G/H$ is \textbf{null} if $\mu_{G/H}(X)=0$, and an arbitrary subset of $G/H$ is \textbf{null} if it is contained in a Borel null set. Since all of such $\mu_{G/H}$ are strongly equivalent, our choice of $\mu_{G/H}$ makes no difference when studying null sets.

Recall that we assume that $G$ is $\sigma$-compact. Then for closed subgroups $H$ and $K$ of $G$,  $H=\bigcup_n X_n$ and $K=\bigcup_m Y_m$ where $X_n$ and $Y_n$ are compact. Then $HK=\bigcup_{n,m} X_n Y_m$ which is Borel. Moreover, in $G/K$, $H/K=\bigcup_n X_n/K$ which is also Borel. Both $HK$ and $H/K$ are Haar measurable. We have a parallel statement to \cite{KP} Proposition 2.3.

\begin{prop}\label{q-lemma}
Let $H$, $K$ be closed subgroups of $G$. Then $HK$ is null in $G$ iff $H/K$ is null in $G/K$.
\end{prop}
\begin{proof}
Approximate the relevant characteristic functions by continuous compactly supported functions, and apply the Weyl formula.
\end{proof}

Now we study the independence relation $\ind^\HH$ induced by $\I^\HH$. 

\begin{theorem}
Let $G$ be locally compact and $\sigma$-compact. Then
$(G, \ind^\HH)$ is a strongly admissible geometry.
\end{theorem}
\begin{proof}
It suffices to show the last two properties in Definition \ref{admissible}.

(Transitivity) Let $K_1>K_2>K_3$ and $H<K_1$ be closed subgroups of $G$. Note that $HK_3\cap K_2=(H\cap K_2)K_3$ and in $K_1/K_2$, $H/K_2=HK_2/K_2=HK_3/K_2$. The result follows immediately by decomposing 
$G$ into $K_2$ and $G/K_2$, and applying the Weil formula together with Proposition \ref{q-lemma}.

(Dichotomy for closed subgroups)  Let $K<H$ be closed subgroups. It is clear that if $H/K$ is countable, $\mu_H(K)>0$. Now we prove the other direction. Assume, for a contradiction, that $H/K=\{h_\alpha K:\alpha<\lambda\}$ is uncountable and $\mu_H(K)>0$.
Recall that $G$ is $\sigma$-compact and so is every closed subgroup of $G$. Let $(X_n)_{n\in\omega}$ be a sequence of compact subsets of $H$ covering $H$. It is clear that for every $\alpha<\lambda$, there is $j(\alpha)\in\omega$ such that $\mu_H(h_\alpha K\cap X_{j(\alpha)})>0$. Then there is $m\in\omega$ such that $|j^{-1}(m)|=\lambda$. Then there are $\epsilon>0$ and a uncountable set $\Lambda\subset\lambda$ such that for every $\alpha\in \Lambda$, $\mu_H(h_\alpha K\cap X_m)>\epsilon$. Hence, $\mu_H(X_m)>n\epsilon$ for every $n\in\omega$, which is impossible.
\end{proof}

\begin{remark}
Recall that if $G$ is locally compact, it has an open $\sigma$-compact subgroup $U$. We deal with Haar measures on $G$ by ones on $U$ by translations. 

Let $H$ and $K$ be closed subgroups of $G$. Let $g,h\in G$. It is easy to see that if $H\cap gU\neq\emptyset$, then $H\cap gU=x(H\cap U)$ for every $x\in H\cap gU$. Similarly, $K\cap Uh=(K\cap U)y$ for every $y\in K\cap Uh$ when $K\cap Uh$ is nonempty. Then $HK=\bigcup_{g,h} (H\cap gU)(K\cap Uh)=\bigcup_{x,y}x(H\cap U)(K\cap U)y$ where $x\in H\cap gU\neq\emptyset$ and $y\in K\cap Uh\neq \emptyset$. Then it is easy to see that $HK\in \I^\HH_G$ iff $(H\cap U)(K\cap U)\in \I_U^\HH$. So we can study $\I^\HH=\{\I^\HH_H\}_H$ by the ideal scheme $\I^\HH|_U$ on $U$.

Hence, when $G$ is locally compact, $(G, \ind^\HH)$ is an admissible geometry.
\end{remark}

The following example shows that $\ind^\HH\neq \ind^{nm}$ even for compact Polish groups.

\begin{example}
Let $l_n=2^{n+2}$. Let $G_n=\Sym(\{1,\dots,l_n\})$ be the symmetric group on $\{1,\dots,l_n\}$ with the discrete topology. Let $G=\prod_n G_n$. Obviously, $G$ is a compact Polish group.
Let $\mu_G$ be the left Haar measure with $\mu_G(G)=1$.

Let $H_n=\{\sigma\in G_n:\sigma(1)=1\}<G_n$ and $K_n=\{\sigma\in G_n:\sigma(2)=2\}<G_n$. Let $H=\prod_nH_n$ and $K=\prod_n K_n$. It is easy to see that $HK=\prod_n H_nK_n$.

For each $n$, $|H_n|=|K_n|=(l_n-1)!$ and $|H_n\cap K_n|=(l_n-2)!$. Then $|H_nK_n|=\frac{|H_n||K_n|}{|H_n\cap K_n|}=(l_n-1)!(l_n-1)$, so $\frac{|H_n||K_n|}{|G_n|}=\frac{l_n-1}{l_n}=1-\frac{1}{l_n}$.

Note that $\log(1-x)>-2x$ when $0<x<1/2$.
Then $\mu_G(HK)=\prod_n(1-\frac{1}{l_n})=\exp(\sum_n\log(1-2^{-n-2}))>\exp(-\sum_n 2^{-n-1})=e^{-1}>0$. So $HK$ is not null in $G$.

It is clear that $HK$ is compact and hence closed. Moreover, $HK$ has empty interior. So $HK$ is nowhere dense, and hence meager.

If one needs a $G$-structure witnessing the difference, we let $M:=G\sqcup (G/H)\sqcup (G/K)\sqcup(G/(H\cap K))$.
\end{example}

We now drop the assumption that 
$G$ is locally compact and assume instead that 
$G$ is a Polish group. We then attempt to generalize the previous discussion on null sets to Haar null sets. Recall that a Borel subset $X$ of $G$ is \textbf{Haar null} if there exists a Borel probability measure  $\mu$ on $G$ such that  $\mu(gXh)=0$ for all $g,h\in G$; an arbitrary subset of $G$ is \textbf{Haar null} if it is contained in a Haar null Borel subset (\cite{Haar-null} Definition 1.1, which is originally from \cite{C-Haar}). It is known that the collection of Haar null sets forms a $\sigma$-ideal and coincides with $\I^\HH_G$ when $G$ is locally compact. So we also use $\I^\HH_G$ to denote this ideal and use $\I^\HH$ to denote the ideal scheme $\{\I^\HH_H\}_H$ on $G$. 

However, $(G, \ind^\HH)$ is not an admissible geometry in general. Before presenting an example, we recall a fact.

\begin{fact}[\cite{Haar-null} Fact 2.2]\label{compact-cat}
Every compact catcher subset of $G$ is not Haar null.

Here, a subset $X$ of $G$ is called \textbf{compact catcher} if for every compact subset $F\subset G$, there are $g,h\in G$ such that $gFh\subset X$. 
\end{fact}

Now we give an example.

\begin{example}
Let $Y,Z$ be countably infinite sets and $X=Y\sqcup Z$.

For $\sigma\in \Sym(X)$, the symmetric group
on $X$, the support of $\sigma$, $\supp(\sigma):=\{x\in X:\sigma(x)\neq x\}$. We denote $\Sym_\omega(X):=\{\sigma\in \Sym(X):|\supp(\sigma)|<\omega\}$, the subgroup of finitely supported permutations. 

Note that $\Sym_\omega(X)$ is countable.
We equip $\Sym_\omega(X)$ with the discrete topology. Let $G=\Sym_\omega(X)^\omega$ with the product topology. It is clear that $G$ is Polish.

Let $P=\Sym_\omega(Y)^\omega$, $Q=\Sym_\omega(Z)^\omega$ and $R=P\cap Q=\{1_G\}$. 

\begin{claim}
$P$ is compact catcher in $G$.
\end{claim}
\begin{claimproof}
Let $F\subset G$ be compact. Let $F_n<\Sym_\omega(X)$ be the $n$-th projection of $F$. It is clear that $F_n$ is compact and therefore finite. Let $C_n:=\bigcup_{\sigma\in F_n}\supp(\sigma)$. Obviously, $C_n$ is finite. 
Since $Y$ is infinite, there is $\tau_n\in \Sym_\omega(X)$ such that $\tau_n(C_n)\subset Y$. Let $\tau=(\tau_n)\in G$. It is easy to see that $\tau F \tau^{-1}\subset P$.
\end{claimproof}

By Fact \ref{compact-cat}, $P$ is not Haar null in $G$. For the same reason, $Q$ is not Haar null in $G$.  
It is easy to check that $R$ is Haar null in $P$.

We let $K_1=G>K_2=P>K_3=R$ and $H=Q$. Then $(H\cap K_2)K_3\in \I_{K_2}$ but $HK_3\notin \I_{K_1}$ which breaks Transitivity in Definition \ref{admissible}.
\end{example}

Inspired by the model-theoretic characterization of simple theories by forking independence (\cite{Simple} Theorem 6.2.1), we ask:

\begin{ques}
Does there exist a natural (topological and abstract group-theoretic) characterization of those Polish groups $G$ for which $(G, \ind^\HH)$ is an admissible geometry?
\end{ques}

We have shown that $(G,\ind^\HH)$ is an admissible geometry when $G$ is locally compact. But this condition is not necessary. Note that $G=(\Z,+)^\omega$ is an example where $(G,\ind^\HH)$ is an admissible geometry while $G$ is not locally compact. This follows from that the commutativity of $G$, which is an abstract group-theoretic condition, ensures admissibility.

\begin{theorem}
Let $G$ be a commutative Polish group. Then $\ind^\HH=\ind^o=\ind^{nm}$.
\end{theorem}
\begin{proof}
Let $H$, $K$ be closed subgroups of $G$.  Recall that by \cite{KP} Lemma 2.6, $V=HK$ is Borel, so is universally measurable, i.e., measurable with respect to the completion of every Borel probability measure. Then by \cite{C-Haar} Theorem 2, the set $F(V,V):=\{g\in G: gV\cap V\notin \I^\HH_G\}$ is open (possibly empty). Note that $V$ is a subgroup of $G$, since $G$ is commutative.

\begin{claim}
If $V$ is not Haar null, then it is open.
\end{claim}
\begin{claimproof}
Assume that $V\notin \I^\HH_G$.
Obviously, $1_G\in F(V,V)$, so $F(V,V)$ is a nonempty open set. Note that $V=VV^{-1}=\{g\in G: gV\cap V\neq \emptyset\}\supset F(V,V)$. Then $V$ has nonempty interior, so it is open.
\end{claimproof}

\begin{claim}
If $V$ is open, then it is not Haar null.
\end{claim}
\begin{claimproof}
Assume that $V$ is open.
Since $G$ is a Polish group, there is no uncountable family of pairwise disjoint open sets. Thus, $G/V$ is countable. If $V$ is Haar null, then $G$ is Haar null, which is impossible.
\end{claimproof}

From the above two claims, $V\notin \I_G^\HH$ iff $V$ is open. Applying this conclusion to each closed subgroup of $G$, we have $\ind^\HH=\ind^o$.

It is directly from Pettis Theorem (\cite{K-Book} Theorem 9.9) that $V$ is not meager iff $V$ is open. Hence, $\ind^{nm}=\ind^o$.

In summary, we have $\ind^\HH=\ind^o=\ind^{nm}$.
\end{proof}

Combining with Fact \ref{Axioms}, we have
\begin{coro}
Let $G$ be a commutative Polish group. Then $(G,\ind^\HH)$ is an admissible geometry.
\end{coro}

\section{The order of geometries}\label{section-3}

As the NM-rank defined in \cite{KP} Definition 3.11, we define the $\I$-rank $\rk_\I$. 
\begin{definition}
Let $M$ be a $G$-structure, $a\in M^n$, $A\subset M$ be a finite set and $\alpha$ be an ordinal. We define 
$\rk_\I^M(a/A)\geq \alpha+1$ if there is a finite subset $B\subset M$ such that $a\not\ind_A^\I B $ and $\rk_\I^M(a/B)\geq \alpha$. 

We define $\rk_\I^M(a/A)=\sup\{\alpha: \rk_\I^M(a/A)\geq \alpha\}$ and write $\rk_\I^M(a/A)=\infty$ when $\rk_\I^M(a/A)\geq \alpha$ for every ordinal $\alpha$. 

We say that $M$ is \textbf{$\I$-stable} if  $\rk_\I^M(a/A)< \infty$ for every $a$ and $A$.
\end{definition}

In order to dispel any possible confusion, we offer a possibly superfluous explanation.
We always assume that $\rk_\I^M(a/A) \geq 0$. 
Thus, $\rk_\I^M(a/A) =0$ when there is no finite subset $B\subset M$ such that $a\not\ind_A^\I B $.

Note that, as \cite{KP} Proposition 3.12, for a strongly admissible geometry $(G,\ind^\I)$, the $\I$-rank admits Lascar inequalities. These inequalities are useful for studying $\I$-stable $G$-structures, but that is beyond the scope of this note.

Now we give a partial order on the collection of all strongly admissible geometries. To avoid discussing on the empty set, we always assume that $G$ is tame.
For two strongly admissible geometries on $G$ given by ideal schemes $\I$ and $\J$, we write $\ind^\I\leq \ind^\J$ if $\rk_{\I}^M\leq \rk_{\J}^M$ on every $G$-structure $M$. It is obvious that $\ind^\I\leq \ind^\J$ together with $\ind^\I\geq \ind^\J$ implies $\ind^\I= \ind^\J$. The transitivity of 
$\leq$ is immediate, so 
$\leq$ defines a partial order for strongly admissible geometries on $G$. Intuitively, when $\ind^\I \leq \ind^\J$, the relation $\ind^\J$ yields a finer independence relation but imposes a stronger requirement of stability.

For ideal schemes $\I$ and $\J$ on $G$, we write $\I\leq \J$ if $\I_H\subset \J_H$ for every closed subgroup $H$. It follows directly from the definition that $\I\leq \J$ implies $\ind^\I\leq \ind^\J$. More precisely, $\ind^\I\leq \ind^\J$ iff for all closed subgroups $K_1,K_2<H<G$, $K_1K_2\in \I_H$ implies $K_1K_2\in \J_H$. The `` if'' part is obvious and the ``only if'' part can be witnessed by $M:=G\sqcup (G/H)\sqcup (G/K_1)\sqcup (G/K_2)\sqcup(G/(K_1\cap K_2))$.

\begin{prop}
Assume that $G$ is tame.
Then there exists a minimal strongly admissible geometry. 
\end{prop}
\begin{proof}
Let $(\ind^\alpha)_\alpha$ be a downward chain of strongly admissible geometry. Assume that $\ind^\alpha$ is induced by the ideal scheme $\I^\alpha$. Let $\I:=\{\I_H\}$ with $\I_H=\bigcap_\alpha \I_H^\alpha$ for every closed subgroup $H$ of $G$. It is easy to check that $\I$ is strongly admissible and $\I\leq \I^\alpha$ for every $\alpha$. Since every downward chain admits a lower bound, by Zorn lemma, there exists a minimal strongly admissible geometry. 
\end{proof}

It is unknown whether the minimum strongly admissible geometry exists, but it does exist in some well-behaved examples, e.g., nilpotent Lie groups, $\SO(3)$. 

Now we study maximal strongly admissible geometries. In well-behave cases, the maximum strongly admissible geometry exists.
Note that the maximum strongly admissible geometry is the finest independence relation and imposes the strongest requirement of stability. I conjecture that $\ind^{nm}$ is the maximum strongly admissible geometry and therefore intrinsic when $G$ is a Polish group.

\begin{conj}
Let $G$ be a Polish group. Then $\ind^{nm}$ is the maximum strongly admissible geometry. 
\end{conj}

At least we can prove this conjecture under the additional condition that $G$ is $\sigma$-compact.

\begin{theorem}\label{thm-nm-max}
Let $G$ be a $\sigma$-compact Polish group. Then $\ind^{nm}=\ind^o$ is the maximum strongly admissible geometry. 
\end{theorem}
\begin{proof}
Let $\I$ be a strongly admissible ideal scheme on $G$. Let $H,K<L$ be closed subgroups of $G$. Since $G$ is $\sigma$-compact, $H=\bigcup_{n<\omega} C_n$ and $K=\bigcup_{m<\omega}D_m$ with $C_n$ and $D_m$ compact. Then $HK=\bigcup_{n,m} C_n D_m$ where each $C_nD_m$ is compact and hence closed. 

If $HK$ is not meager in $L$, then there are $n,m<\omega$ such that $C_nD_m$ has nonempty interior. Then $HK$ has nonempty interior and it is easy to see that $HK$ is open. Hence, $\ind^{mn}=\ind^o$. Since $G$ is Polish, $G$ can be covered by countably many left translates of $HK$. As $\I_L$ is an $L$-invariant $\sigma$-ideal, we have $HK\notin \I_L$. 

Hence, we have shown that $HK\in\I_L$ implies $HK$ being meager in $L$. It follows directly from the definition that $\ind^\I\leq \ind^{nm}$.
\end{proof}

We now specialize to the case of Lie groups. Note that our Lie groups are second countable, finite dimensional and over the reals. Let us classify strongly admissible geometries for the compact Lie group $G=\SO(3)$. This provides an example where there are a lot of pairwise incomparable strongly admissible geometries.

\begin{example}
Let $G$ be the special orthogonal group $\SO(3)$. 

For $u\in \RP^2$, an unoriented axis of $\R^3$, we let $T_u$ be the subgroup of all rotations about the axis $u$. It is easy to see that $T_u\cong \SO(2)$. Let $N_u$ be the normalizer of $T_u$ in $G$. Then $N_u/T_u\cong\Z/2\Z$ and $N_u\cong \mathrm{O}(2)$.

Since $\mathfrak{so}(3)$, the Lie algebia of $G=\SO(3)$, does not have $2$-dimensional subalgebra, $G$ does not have $2$-dimensional closed subgroups. It is not hard to check that the closed subgroups of $G$ are
\begin{itemize}
\item ($0$-dimensional) finite subgroups;

\item ($1$-dimensional) $T_u$ and $N_u$ for $u\in \RP^2$;

\item ($3$-dimensional) $G$ itself.
\end{itemize}

Let $\I$ be a strongly admissible geometry for $G$. Obviously, $\I_H=\{\emptyset\}$ when $H$ is finite.
Moreover, when $H=T_u$ or $N_u$, it is easy to see that for closed subgroups $K_1,K_2<H$, $K_1K_2\in \I_H$ iff $K_1K_2$ is finite iff $K_1K_2$ is meager iff it is Haar null. Although there are many choices of the ideal on $H$, they all give the same induced geometry. Hence, the key point is to determine the ideals on $G=\SO(3)$. Since $N_u/T_u\cong\Z/2\Z$, we only need to check whether $T_uT_v$ belongs to $\I_G$ for $u\neq v\in \RP^2$.

For $u\neq v\in \RP^2$, $T_u\cap T_v=\{1_G\}$ and $T_u\times T_v\rightarrow \SO(3):(s,t)\mapsto st$ is an embedding. So $T_uT_v$ is a compact $2$-dimensional surface.
Moreover, we let $\theta=\angle(u,v)$ be the angle between the axes $u$ and $v$. Note that we always set $\theta\in(0,\frac{\pi}{2}]$. Let $p,q$ on the unit sphere $S^2$ with $u=[p]$, $v=[q]$ and inner product of $p$ and $q$ in $\R^3$, $\langle p,q \rangle=\cos\theta$. Then it is easy to check that 
$$T_uT_v=\{g\in G: \langle p,gq \rangle=\cos \theta \}.$$
Moreover, one can see that every left-right translate of $T_uT_v$ has the from 
$$S_{p',q',c}:=\{g\in G: \langle p',gq' \rangle= c \}$$
where $p',q'\in S^2$ and $|c|=\cos \theta$.

\begin{claim}
Let $p,q,p'q'\in S^2$ with $[p]\neq [q]$ and $[p']\neq [q']$. Let $c,c'\in (-1,1)$. Assume that $S_{p,q,c}\cap S_{p',q',c'}$ has nonempty interior in $S_{p,q,c}$ (or $S_{p',q',c'}$). Then $[p]=[p']$, $[q]=[q']$ and $|c|=|c'|$.
\end{claim}
\begin{claimproof}
Fix $g_0\in S_{p,q,c}$. Then $S_{p,q,c}$ can be parameterized by $T_{[p] }\times T_{[q]}$ as $S_{p,q,c}=\{sg_0t:s\in T_{[p]}, t\in T_{[q]}\}$. Let $$\chi: T_{[p] }\times T_{[q]}\rightarrow \R: (s,t)\mapsto \langle p', sg_0tq' \rangle.$$
It is clear that when $sg_0t\in S_{p',q',c'}$, $\chi(s,t)=c'$. Since $\chi$ is a real analytic function, $\chi\equiv c'$ when $\chi=c'$ on an open set. Hence, when $S_{p,q,c}\cap S_{p',q',c'}$ has nonempty interior in $S_{p,q,c}$, we have $S_{p,q,c}=S_{p',q',c'}$ which implies $[p]=[p']$, $[q]=[q']$ and $|c|=|c'|$.
\end{claimproof}

Hence, by the Baire Category Theorem, $T_uT_v$ can not be covered by countably many left-right translates of $T_{u_n}T_{v_n}$ with $\angle(u_n,v_n)\neq \angle(u,v)$.

For given $\I$, we let $$\Theta_\I:=\{\theta\in(0,\frac{\pi}{2}]: T_u T_v\in \I_G \text{ for some (all) } u,v \text{ with } \angle(u,v)=\theta \}$$
Here, ``some'' and ``all'' are equivalent, because $\I_G$ is invariant by conjugations. 

Conversely, given $\Theta\subset(0,\frac{\pi}{2}]$, let 
$\I_G^\Theta$ be the hereditary $\sigma$-ideal generated by 
\begin{itemize}
\item left-right cosets of proper closed subgroups;
\item left-right translates of $T_uT_v$ for $u\neq v\in \RP^2$ with $\angle(u,v)\in\Theta$.
\end{itemize}

The previous analysis on $T_uT_v$  shows that $T_u T_v\in \I_G^\Theta$ iff $\angle(u,v)\in \Theta$. Combining $\I_G^\Theta$ and the unique strongly admissible geometries on proper subgroups, we give an ideal scheme $\I^\Theta$ for $G$. It is easy to check that $\I^\Theta$ is strongly admissible, and we leave the details to the readers.

In summary, we have an order-preserving one-to-one correspondence:
$$(\{\text{Strongly admissible geometries $\ind^\I$ for } \SO(3) \},\leq) \overset{\I\mapsto \Theta_\I}{\underset{\I^\Theta\mapsfrom\Theta}{\rightleftarrows}}(\mathcal{P}((0,\frac{\pi}{2}]),\subset).$$
Then $\ind^{\I^\emptyset}$ is the minimum and $\ind^{\I^{(0,\frac{\pi}{2}]}}=\ind^{nm}=\ind^\HH$ is the maximum. Moreover, we can find a collection of pairwise incomparable strongly admissible geometries of size $2^\mathfrak{c}$.
\end{example}

We now present another example. The proof is analogous, and we leave the details to the readers.

\begin{example}\label{se2}
Let $G$ be the special Euclidean group of the plane $\mathrm{SE(2)}\cong \R^2\rtimes \SO(2)$.

Since our ideals are $\sigma$-ideals, it suffices to study connected closed subgroups. One can check that the proper connected closed subgroups of $G$ are
\begin{itemize}
\item ($0$-dimensional) $\{1_G\}$;

\item ($1$-dimensional) $T_p$, the subgroup of all rotations about the point $p$, for $p\in \R^2$;

\item ($1$-dimensional) $P_l$, the subgroup of all translations along $l$ , for $l\in \RP^1$, an unoriented axis in $\R^2$;

\item ($2$-dimensional) $V\cong (\R,+)^2$, the subgroup of translations.
\end{itemize}

By checking or using Theorem \ref{nil-unique} below, we know that $\ind^{nm}$ is the unique strongly admissible geometry on these proper subgroups. So the strongly admissible geometries of $G=\mathrm{SE}(2)$ are determined by the ideals on $G$.

Given $\Delta\subset(0,\infty)$, we let $\I^\Delta_G$ be the hereditary $\sigma$-ideal generated by

\begin{itemize}
\item left-right cosets of proper closed subgroups;

\item line-rotation surfaces, i.e., left-right translates of $T_p P_l$ and $P_lT_p$ for $p\in \R^2$ and $l\in \RP^1$;

\item picked rotation-rotation surfaces, i.e, left-right translates of $T_pT_q$ for $p,q\in \R^2$ with $|p-q|\in \Delta$.
\end{itemize}

Let $\I^\Delta$ be the ideal scheme given by $\I_G^\Delta$. It is not hard to check that $\I^\Delta$ is strongly admissible and $T_pT_q\in \I_G^\Delta$ iff $|p-q|\in \Delta$.

Conversely, given a strongly admissible ideal $\I$, we let
$$\Delta_\I:=\{\delta\in (0,\infty): T_pT_q\in \I_G \text{ for some (all) } p,q \text{ with } |p-q|=\delta\}$$

Then we have an order-preserving one-to-one correspondence:
$$(\{\text{Strongly admissible geometries $\ind^\I$ for } \mathrm{SE}(2)\},\leq) \overset{\I\mapsto \Delta_\I}{\underset{\I^\Delta\mapsfrom\Delta}{\rightleftarrows}}(\mathcal{P}((0,\infty)),\subset).$$

\end{example}

Let us end this note by showing the uniqueness of the strongly admissible geometry for nilpotent Lie groups. First, note that since our Lie group $G$ is second countable, its connected component $G^0$ is of countable index in $G$. As our ideals are $\sigma$-ideals, every ideal scheme $\I$ on $G$ is determined by $\I|_{G^0}$ on $G^0$. Hence, It suffices to consider connected Lie groups. Also note that every closed subgroup of a nilpotent Lie group is nilpotent. In order to proceed by induction on the dimension, we give a result for nilpotent Lie groups.

\begin{lemma}
Let $G$ be a connected nilpotent Lie group and $K$ a connected closed proper  subgroup. Then there is a connected closed normal subgroup $N$ such that $K<N\lhd G$ and $G/N$ is a $1$-dimensional connected Lie group.
\end{lemma}

\begin{proof}
First, we show that there is a connected closed normal proper subgroup of $G$ containing $K$. Let $\mathfrak{g}$ and $\mathfrak{k}$ be the Lie algebras of $G$ and $K$ respectively. Note that the action of $\mathfrak{k}$ on $\mathfrak{g}/\mathfrak{k}$ is ad-nilpotent, and by Engel Theorem, there is $X\in \mathfrak{g}$ such that $0\neq X+\mathfrak{k}\in \mathfrak{g}/\mathfrak{k}$ is killed by $\mathfrak{k}$, and equivalently, $[X,\mathfrak{k}] \subset \mathfrak{k}$. Hence, the normalizer
$N_{\mathfrak{g}}(\mathfrak{k}):=\{X\in\mathfrak{g}:[X,\mathfrak{k}]\subset \mathfrak{k}\}$ is strictly larger than $\mathfrak{k}$. Note that the Lie algebra of $N_G(K)^0$ is $N_{\mathfrak{g}}(\mathfrak{k})$, so $K<N_G(K)^0$ with $\dim(K)<\dim(N_G(K)^0)$. Since $\dim (G)$ is finite, this process always terminates after finitely many repetitions. Then we can find a connected closed subgroup $N$ of $G$ containing $K$ such that $\dim(N)<\dim(G)$ and $N_G(N)^0=G$, which implies that $N$ is a proper normal subgroup of $G$.

Now we assume in addition that $N$ is the one with the maximum dimension. Let $Q=G/N$. Note that $Q$ is a connected nilpotent Lie group. Assume that $\dim(Q)\geq 2$. If $Q$ is not commutative, then the center is a nontrivial closed normal subgroup; if $Q$ is commutative, it is a product of copies of $(\R,+)$ and torus $\R/\Z$, so obviously it has a nontrivial closed normal subgroup. Let $H$ be a  connected closed normal nontrivial subgroup of $Q$. Then the preimage of 
$H$ under the quotient map 
$G\rightarrow G/N=Q$ would violate the maximality of 
$N$.
\end{proof}

Now we prove the theorem.

\begin{theorem}\label{nil-unique}
Let $G$ be a nilpotent Lie group. Then $\ind^{mn}=\ind^\HH$ is the unique strongly admissible geometry for $G$. 
\end{theorem}
\begin{proof}
We may assume $G$ is connected and restrict our attention to connected closed subgroups. Induction on $\dim(G)$.

Let $\I$ be a strongly admissible ideal scheme for $G$. By induction, for every closed subgroup $L$ of $G$, $\I|_H$ induces $\ind^{nm}=\ind^\HH$ on $L$. So it suffices to study $\I_G$.  Let $H,K$ be connected closed proper subgroups such that $HK$ is meager in $G$. It suffices to show that $HK\in \I_G$.

Let $S$ be the closure of the group generated by $HK$. It is clear that $S$ is a closed subgroup of $G$. If $S$ is a proper subgroup, then $G/S$ is uncountable, which implies that $S\in \I_G$ and $HK\in \I_G$. So we assume that $S=G$.

By the previous lemma, we let $N$ be a  connected closed normal subgroup of $G$ such that $K<N$ and $\dim(G/N)=1$. 

\begin{claim}
$(H\cap N)K$ is meager in $N$.
\end{claim}
\begin{claimproof}
Otherwise, by Theorem \ref{thm-nm-max}, $(H\cap N)K$ is open in $N$. Note that $HK=H(H\cap N)K$ and the map $\chi: H\times N\rightarrow G:(h,n)\mapsto hn$ is an open map. Then $HK=\chi(H\times (H\cap N)K)$ is open in $G$, contradicting the assumption that $HK$ is meager in $G$.
\end{claimproof}

Since $\I|_N$ induces $\ind^{mn}$ on $N$, we have  $(H\cap N)K\in \I_N$. Taking $K_1=G>K_2=N>K_3=K$, Transitivity says that $(H\cap N)K\notin \I_N$ and $HN\notin \I_G$ iff $HK\notin \I_G$. Hence, we have $HK\in \I_G$.
\end{proof}

\begin{remark}
The conclusion cannot be extended to solvable Lie groups.  The group $\mathrm{SE}(2)$ provides a counterexample  (Example \ref{se2}). 
\end{remark}

\end{document}